\documentclass{amsart}
\usepackage{epsfig}
\usepackage{amssymb}
\usepackage{graphicx}

\newcommand\C{\mathbb{C}}
\newcommand\Q{\mathbb{Q}}

\newcommand\hyp{\mathbb{H}}
\newcommand\ring{\mathcal{O}}
\newcommand\pr{\mathcal{P}}
\newcommand\pri{\mathcal{Q}}

\newcommand\Z{\mathbb{Z}}
\newcommand\F{\mathbb{F}}
\newcommand{\Fb} {\bar{\F}}
\newcommand\PSL{\mathop{\rm PSL}\nolimits}
\newcommand\SL{\mathop{\rm SL}\nolimits}
\newcommand\SLC{{\rm SL}_2(\C)}

\newtheorem{theorem}{Theorem}[section]
\newtheorem{lemma}[theorem]{Lemma}
\newtheorem{proposition}[theorem]{Proposition}
\newtheorem{corollary}[theorem]{Corollary}

\newtheorem{remark}[theorem]{Remark}
\newtheorem{conjecture}[theorem]{Conjecture}

\newcounter{nootje}
\title{Commensurability classes of $(-2,3,n)$ pretzel knot complements}
\author{Melissa L.\ Macasieb and Thomas W.\ Mattman}
\address{Department of Mathematics, The University of British Columbia,
	 Room 121, 1984 Mathematics Road,
Vancouver, B.C., Canada V6T 1Z2}
\email{macasieb@math.ubc.ca}
\address{Department of Mathematics and Statistics,
         California State University, Chico,
         Chico, CA 95929-0525, USA}
\email{TMattman@CSUChico.edu}

\thanks{Much of this research occured during a visit of the second author to UBC. He would like to thank Dale Rolfsen and the department for their hospitality.}

\begin{document}

\begin{abstract}
Let $K$ be a hyperbolic $(-2,3,n)$ pretzel knot and $M = S^3 \setminus K$ its complement.
For these knots, we verify a conjecture of Reid and Walsh: there are at most three knot complements in the commensurability class of $M$. Indeed, if $n \neq 7$, we show that $M$ is the unique knot complement in its class. We include examples to illustrate how our methods apply to a broad class of Montesinos knots. 
\end{abstract}

\maketitle
\section{Introduction}
Two hyperbolic 3-manifolds $M_1 = \hyp^3/ \Gamma_1$ and $M_2 = \hyp^3/ \Gamma_2$ are {\em commensurable} if they have homeomorphic finite-sheeted covering spaces.  On the level of groups, this is equivalent to $\Gamma_1$ and a conjugate of $\Gamma_2$ in $\rm{Isom}(\hyp^3)$ sharing some finite index subgroup.  The {\em commensurability class} of a hyperbolic 3-manifold M is the set of all 3-manifolds commensurable with $M$. 

Let $M = S^3 \setminus K = \hyp^3/\Gamma_K$ be a hyperbolic knot complement. A conjecture
of Reid and Walsh suggests that the commensurability class of $M$ is a strong knot
invariant:
\begin{conjecture}[\cite{RW}]\label{conjRW}%
Let $K$ be a hyperbolic knot. Then there are at most three knot complements in
the commensurability class of $S^3 \setminus K$.
\end{conjecture}
Indeed, Reid and Walsh prove that for $K$ a hyperbolic $2$--bridge knot, $M$ is the 
only knot complement in its class. This may be a wide-spread phenomenon; by combining
Proposition~4.1 of \cite{RW} with their proof of Theorem~5.3(iv), we have the following set of 
sufficient conditions for $M$ to be alone in its commensurability class.
\begin{theorem}\label{thmRW}%
Let $K$ be a hyperbolic knot in $S^3$. If $K$ admits no hidden symmetries, has no lens space surgery, and admits either no symmetries or else only a strong inversion and no other symmetries, then $S^3 
\setminus K$ is the only knot complement in its commensurability class.
\end{theorem}

\begin{figure}
\centerline{
\epsfig{file=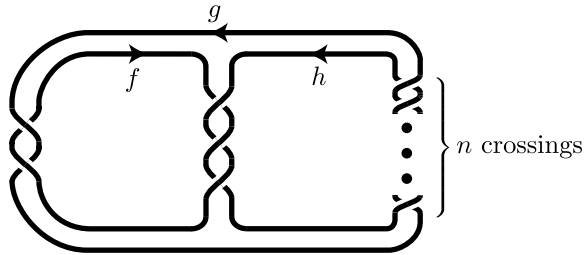}}
\caption{The $(-2,3,n)$ pretzel knot}
\label{figpretz1}
\end{figure}

The {\em $(-2, 3, n)$ pretzel knot}, $n \in \Z$, is defined by the diagram in Figure \ref{figpretz1}.  The diagram determines a knot when $n$ is odd and determines a link otherwise.  Moreover, these knots have complements that are hyperbolic precisely when $n \neq 1, 3, 5$.  In fact, both this family of knots and the family of 2-bridge knots are part of the larger family of Montesinos knots.   Our main result is the following theorem:

\begin{theorem}\label{thmgrand}
Let $K$ denote a hyperbolic $(-2,3,n)$ pretzel knot. The conjecture of Reid and Walsh holds for $K$.
Moreover, unless $n = 7$, $S^3 \setminus K$ is the only knot complement in its commensurability class.
\end{theorem}
This will follow from Theorem~\ref{thmRW} in the case $n \neq 7$.  As for the $(-2,3,7)$ pretzel knot, Reid and Walsh show that there are exactly two other knot complements in the commensurability class of its complement which correspond to its two lens space surgeries.

Taking advantage of what is already known about these knots,  we can reduce 
Theorem~\ref{thmgrand} to the following theorem:
\begin{theorem}\label{thmain}%
A hyperbolic $(-2,3,n)$ pretzel knot admits no hidden symmetries.
\end{theorem}
Indeed, let $K_n$ be the $(-2,3,n)$ hyperbolic pretzel knot, i.e., $n$ is odd and $n \neq 1,3,5$.
Assuming in addition that $n \neq 7$, then $K_n$ admits no non-trivial cyclic surgeries \cite{M} and,
therefore, no lens space surgeries. The $(-2,3,1)$ knot does have symmetries other than a 
strong inversion, but it is a $2$-bridge knot and therefore covered by the work of Reid and Walsh \cite{RW}.
Assuming $K_n$ is a hyperbolic knot and $n \neq -1$,
then $K_n$ is strongly invertible and has no other symmetries \cite{BZ,S}.
Thus, Theorem~\ref{thmgrand} follows once we prove Theorem~\ref{thmain}.

The main part of this paper, then, is devoted to proving Theorem~\ref{thmain}.
Using work of Neumann and Reid~\cite{NR} this comes down to arguing that the invariant trace field 
of the $(-2,3,n)$ pretzel knot, $k_n$, has neither $\Q(i)$ nor $\Q(\sqrt{-3})$ as a subfield. We will see that it suffices to show this in the case
where $n$ is negative. Indeed, we conjecture the following:
\begin{conjecture}\label{conj}
Let $n$ be an odd, negative integer. Then the $(-2,3,n)$ and $(-2,3,6-n)$ pretzel knots have the same trace field.
\end{conjecture}
Using a computer algebra system, we have verified this conjecture for
$-49 \leq n \leq -1$. Note that the complements of $(-2,3,n)$ and $(-2,3,6-n)$ also share the same volume. This is stated in Week's thesis~\cite{W} (see also~\cite{BH}) and a new proof by Futer, Schleimer, and Tillman has recently been announced~\cite{F}.
This suggests that the $(-2,3,n)$ pretzel knots provide an
infinite set of examples of pairs of hyperbolic knot complements that
share the same volume and trace field and yet are not commensurable.
Our proof of Theorem~\ref{thmain} does not depend on the validity of our conjecture.

The results we have just quoted, \cite{BZ,M, S}, show that many other Montesinos knots
also have no lens space surgeries and admit, at most, a strong inversion. 
So, our methods apply to a large class of Montesinos knots.  


Our paper is organised as follows. In the next two sections we review some definitions and results that are necessary in our arguments; 
we also present evidence in support of Conjecture~\ref{conj} and prove Theorem~\ref{thmain}.
The argument comes down to showing $\Q(i)$ is not a subfield of the trace
field (Section~\ref{secQi}) and neither is $\Q(\sqrt{-3})$ (Section~\ref{secQ3}).  In Section~\ref{secMont}, we extend our results to $(p,q,r)$ pretzel knots and discuss how they apply to Montesinos knots in general.

\section{Hidden Symmetries, the Trace Field, and the Cusp Field}
In this section, we explicitly describe the relationship between hidden symmetries of a hyperbolic knot complement and its trace field.   Although some of our definitions will be phrased in terms of hyperbolic knot complements, they apply to the more general class of Kleinian groups of finite covolume.

Let $S^3 \setminus K$ be a hyperbolic knot complement and $\pi_1(S^3 \setminus K)$ its fundamental group.  Then $S^3 \setminus K$ is homeomorphic to $\hyp^3/ \Gamma_K$, for some discrete torsion free subgroup $\Gamma_K$ of $\rm{Isom}(\hyp^3) = \PSL_2(\C)$.  By the Mostow-Prasad Rigidity Theorem, $\Gamma_K$ is unique up to conjugacy if $K$ is hyperbolic and has finite volume.  Since $\pi_1(S^3 \setminus K)$ is a knot group, the isomorphism from $\pi_1(S^3 \setminus K)$ onto $\Gamma_K$ lifts to an isomorphism $\rho_0:\pi_1(S^3 \setminus K) \to \SL_2(\C)$, which is usually called the {\em discrete faithful representatation} of $\pi_1(S^3 \setminus K)$.   We will now abuse notation and identify $\pi_1(S^3 \setminus K)$ with its image $\Gamma_K \subset (\rm{P})\SL_2(C)$ via the discrete faithful representation.

The commensurator of a group $\Gamma \subset \PSL_2(\C)$ is the group 
$$C(\Gamma)=\{g \in \Gamma : |\Gamma : \Gamma \cup g^{-1}\Gamma  g| < \infty \}.$$ If $C^{+}(\Gamma)$ denotes the subgroup of orientation-preserving isometries of $C(\Gamma)$, then $K$ is said to have {\em hidden symmetries} if $C(\Gamma)$ properly contains the normalizer of $\Gamma$ in $\PSL_2(\C)$.  

Recall that the trace field of $\Gamma$, $\mbox{tr } \Gamma = \{\mbox{tr}\gamma : \gamma \in \Gamma  \}$, is a simple extension of $\Q$ and the  {\em invariant trace field} of $\Gamma$, $k \Gamma  =\{\mbox{tr}(\gamma^2) : \gamma \in \Gamma\}$, is a subfield of the trace field that is an invariant of the commensurability class of $\Gamma$.  In the case $\Gamma= \Gamma_K$ corresponds to the fundamental group of a hyperbolic knot complement, these two fields coincide.  After conjugating, if necessary, one can arrange that a peripheral subgroup of $\Gamma$ has the form
$$\langle 
\left(
\begin{array}{cc}
  1& 1    \\
  0& 1   
\end{array}
\right), \left(
\begin{array}{cc}
  1& g    \\
  0& 1   
\end{array}
\right)
\rangle.$$
The element $g$ is called the cusp parameter of $\Gamma$ and the field $\Q(g)$ is called the {\em cusp field} of $\Gamma$.  One can show that $g \in k\Gamma$ (see for example \cite[Proposition 2.7]{NR}).  Therefore, the cusp field is a subfield of the trace field. 

The following corollary of \cite[Proposition 9.1]{NR}, relates the existence of hidden symmetries of $K$ to the cusp field of $\Gamma$:

\begin{corollary}[\cite{RW}] \label{hidden} Let $K$ be a hyperbolic knot with hidden symmetries.  Then the cusp parameter of $S^3 \setminus K$ lies in $\Q(i)$ or $Q(\sqrt{-3})$.
\end{corollary}

\section{The Trace Field of the $(-2,3,n)$ Pretzel Knot  \label{secprelim}}

In this section we determine the trace field $k_n$ of the $(-2,3,n)$ pretzel knot, prove Theorem~\ref{thmain}, and provide 
evidence in support of Conjecture~\ref{conj}.

Let $K_n$ denote the $(-2,3,n)$ pretzel knot. As above, 
we'll assume $n$ is odd and $n \neq 1,3,5$ so that $K_n$ is a hyperbolic knot. As described in the previous section, there is a discrete faithful $(\rm{P})\SLC$-representation $\rho_0$ of the knot group $\Gamma_{n}$:
$$\Gamma_n := \pi_1(S^3 \setminus K_n) \cong 
\langle f,g, h \mid hfhg  = fhgf, gf (hg)^{(n-1)/2} = f (hg)^{(n-1)/2} h \rangle $$
where the generators $f$, $g$, and $h$ are as indicated Figure \ref{figpretz1}.

To determine the trace field, we'll need to describe the parabolic representations $\rho$ of $\Gamma_n$. The generators $f$, $g$, $h$ must be mapped to 
conjugate elements of trace two. Thus, after an appropriate conjugation in $\SLC$, 
we may assume (cf. \cite{R}),
\begin{equation}\label{eqfgh}
\rho(f) = \left(
\begin{array}{cc} 
1-uv & -v^2 \\
u^2 & 1+uv 
\end{array} \right), 
\rho(g) = \left(
\begin{array}{cc} 
1 & 0 \\
w & 1 
\end{array} \right),
\mbox{ and } \rho(h) = \left(
\begin{array}{cc} 
1 & -1 \\
0 & 1 
\end{array} \right).
\end{equation}
Taking the trace of $\rho(hfhg) - \rho(fhgf)$, we have
$(u-v-1)(u-v+1)w=0$.
If $w = 0$, the representation will not be faithful, so
we must have $u = v \pm 1$.
As either choice will lead to the same field $k_n$, we'll 
set $u = v + 1$. Then the upper left entry of $\rho(hfhg) - \rho(fhgf)$
becomes $v^2(vw-(v+1)(v+2))$. Again, $v = 0$ would mean $\rho$
is not faithful (for example, it would follow that $\rho([f,g]) = I$)
so we can set $w = (v+1)(v+2)/v$. Then, in order for
$\rho$ to be a representation of $\Gamma_n$, the second
relation implies $v$ must satisfy the 
polynomial $p_n$ defined by the following recurrences.
If $n$ is odd and negative,
\begin{eqnarray*}
p_{-1} &  = & v^3 + 2v^2 + v +1 \\
p_{-3} & = & -(v^5 + 3 v^4 + 4 v^3 + 5 v^2 + 4 v + 2) \\
p_{n} & = & - ((v^2+v+2)p_{n+2} +v^2 p_{n+4}) \mbox{ for } n < -3,
\end{eqnarray*}
while if $n$ is odd and at least $7$,
\begin{eqnarray*}
p_{7} &  = & -(v^3 + 2v^2 + 8v + 8) \\
p_{9} & = &  v^5 + 4 v^4 + 10 v^3 + 16 v^2 + 24 v + 16 \\
p_{n} & = & - ((v^2+v+2)p_{n-2} +v^2 p_{n-4}) \mbox{ for } n > 9.
\end{eqnarray*}
It follows that the discrete faithful representation $\rho_0$ corresponds to a root $\alpha_n$ of (some irreducible factor of)
$p_n$.  Moreover, the trace field $k_n = \Q(\alpha_n)$. 

In this paper we will restrict attention to $n$ negative and
use $p_n$ to argue that $k_n$ has neither $\Q(i)$ nor $\Q(\sqrt{-3})$ 
as a subfield.
However, an easy induction shows that, for $n$ odd and negative,
$p_{n}(v) = v^{2-n}2^{(n+1)/2}p_{6-n}(2/v)$.  This shows that both $k_n$ and $k_{6-n}$ correspond to factors of the same polynomial.  Therefore, 
our methods will imply that the same conclusion holds for $n$ positive:
$k_n$ has neither $\Q(i)$ nor $\Q(\sqrt{-3})$ as a subfield when $n 
\geq 7$ is odd. This is why we can restrict our attention to the case where $n$ is negative.

Before proving Theorem~\ref{thmain} we introduce
another family of polynomials under the assumption that $n$ is odd and negative:
\begin{eqnarray*}
q_{-1} &  = & w^3 -w^2 + 2 w - 7\\
q_{-3} & = & w^5 - 2 w^4 - 2 w^3 + 5 w^2 + 3 w -9\\
q_{-5} & = & w^7 -2 w^6  - 4 w^5 + 8 w^4 + 4 w^3 - 7 w^2 + 2 w -7 \\
q_{n} & = & (w^2-1)(q_{n+2} - q_{n+4}) + q_{n+6} \mbox{ for } n < -5.
\end{eqnarray*}
As the following lemma shows, these polynomials are related
to the polynomials $p_n$ defined above by letting
$w = 2 - (v+1)(v+2)/v$.   
\begin{lemma}
Let $n$ be a negative, odd integer.  Then 
$$q_n(w) = \frac{p_n(v) p_{6-n}(v)}{ v^{2-n}}$$
where $w = 2 - (v+1)(v+2)/v$. 
\end{lemma}

\begin{proof}

It is easy to verify the equality for $n = -1$, $-3$, $-5$. Let $n < -7$.
Under the substitution $w = 2-(v+1)(v+2)/v$, $w^2-1$ becomes
$(v^2+2v+2)(v^2+2)/v^2$. Thus, using induction,
\begin{eqnarray*}
q_n(w)&=& q_n \left( 2-(v+1)(v+2)/v\right) \\
& = & \frac{(v^2+2v+2)(v^2+2)}{v^2} 
\left( \frac{p_{n+2} p_{6-(n+2)}}{v^{2-(n+2)}} - \frac{p_{n+4} p_{6-(n+4)}}{v^{2-(n+4)}} \right) +
\frac{p_{n+6} p_{6-(n+6)}}{v^{2-(n+6)}}\\
& = & \frac{1}{v^{2-n}} ((v^2+v+2) p_{n+2} + v^2 p_{n+4})
((v^2+v+2) p_{6-(n+2)} + v^2 p_{6-(n+4)})\\
& = & \frac{p_{n}(v) p_{6-n}(v)}{ v^{2-n}}
\end{eqnarray*}
\end{proof} 
This shows that $k_n= \Q(\alpha_n) \cong \Q(\beta_n)$, where $\alpha_n$ and $\beta_n$ are  
roots of $p_n$ and $q_n$, respectively.

In the next two sections we will prove the following two propositions using the polynomials $p_n$ and $q_n$ defined above.  As we have mentioned,
because of the connection between the $p_n$ with $n$ negative and with $n$ positive, it
will suffice to make the argument in the case that $n$ is a negative, odd integer.

\begin{proposition}\label{noimag} Let $K_n$ denote the $(-2,3,n)$ pretzel knot with trace field $k_n$ where $n$ is odd and negative.  
Then $\Q(i)$ is not a subfield of $k_n$.
\end{proposition}

\begin{proposition}\label{nosqrt3} Let $K_n$ denote the $(-2,3,n)$ pretzel knot with trace field $k_n$ where $n$ is odd and negative.  
Then $\Q(\sqrt{-3})$ is not a subfield of $k_n$.
\end{proposition}

Assuming these two results, we can prove Theorem \ref{thmain}.

\begin{proof}[Proof of Theorem \ref{thmain}]  Let $K_n$ denote the $(-2,3,n)$ pretzel knot with $n$ odd and $k_n$ its trace field.  By assumption $K_n$ is hyperbolic, so $n \neq 1, 3, 5$ and by the remarks above it suffices to consider $n<0$.  It follows from the preceding two propositions that $k_n$ contains neither $\Q(i)$ nor $\Q(\sqrt{-3})$ if $n<0$.  Therefore, by Corollary \ref{hidden}, $K_n$ has no hidden symmetries for all $n \neq 1, 3, 5$.
\end{proof}

As for Conjecture~\ref{conj}, it would follow from the following:
\begin{conjecture}\label{conjirr}
If $n$ is odd and negative,  then $p_n$ and $q_n$ are irreducible.
\end{conjecture}
We have verified 
Conjecture~\ref{conjirr} for $-49 \leq n \leq -1$, using a 
computer algebra system.
The conjecture has two other important consequences.
\begin{remark}
If we could prove Conjecture~\ref{conjirr}
for every $n$, we could immediately deduce that $k_n$ has no $\Q(i)$ nor $\Q(\sqrt{-3})$ subfield. Indeed,
as these polynomials have odd degree, $k_n$ would then be an odd degree extension of $\Q$ and therefore
would admit no quadratic subfield.
\end{remark}
\begin{remark}
Conjecture~\ref{conjirr} would imply that the trace field of
the $(-2,3,n)$ (and $(-2,3,6-n)$) pretzel knot has degree $2-n$. This 
agrees with an observation of Long and Reid~\cite[Theorem 3.2]{LR} that the degree of the trace fields of manifolds obtained by
Dehn filling a cusp increases with the filling coefficient.
(Hodgson made a similar observation. See also~\cite{HS}, especially Corollary~1 and the Question that follows it.)
\end{remark}

\section{$\Q(i)$ is not a subfield of $k_n$ \label{secQi}}

In this section, we prove Proposition \ref{noimag}.  Our main tool is the recursion defining the polynomials $q_n \in \Z[w]$ ($n$ negative, odd) and their reduction modulo $2^k$, for $k$ a positive integer.

\begin{proposition}\label{prop}  Let $q_n \in Z[w]$ be as described in the previous section.  Then 
$$
q_n(w) \equiv  (w+1)^e \prod_{i=2}^m g_i(w) \bmod 2.
$$
where the $g_i$ are relatively prime and $\mbox{\rm deg } g_i \geq 2$ for all $2 \leq i \leq m$ and $e = 0$ (resp. 2) if $3 \nmid n$ (resp. $3 | n$).

\end{proposition}
\begin{proof}  The recursion relation gives 
$$
q_n' \equiv (w+1)^2 (q_{n+2}'-q_{n+4}')+q_{n+6}' \bmod 2.
$$
By induction, one can show that $q_n - w q_n' \equiv (w+1)^2 \bmod 2$.  Therefore, $\mbox{gcd}(q_n,q_n')\equiv \mbox{gcd}(q_n',(w+1)^2) \equiv 1, (w+1),$ or $(w+1)^2 \bmod 2$.  Also, by induction,
$(w+1)$ is a factor of $q_n' \bmod 2$ if and only if $(w+1)$ is a factor of $q_{n+6}'$.   When $n$ is not a multiple of 3, $(1+w)$ is a not factor of $q_n'$ and so $\mbox{gcd}(q_n,q_n') \equiv  1 \bmod 2$.  This shows that $q_n \bmod 2$ has no repeated factors in the case $3 \nmid n$.   When $n$ is a multiple of 3, $(w+1)$ is a factor of $q_n'$ and $\mbox{gcd}(q_n,q_n')\equiv (w+1)^f \bmod 2$, where $f \leq 2$.   Suppose that $3 | n$.  Let $e$ be the greatest integer such that  $(w+1)^e$ divides $q_n \bmod 2 $.  If $e>3$, then $(w+1)^3$ divides $q_n' \bmod 2$, which implies $\mbox{gcd }(q_n,q_n')$ is divisible by $(w+1)^3 \bmod2$, which is a contradiction.  Therefore, $e=2$ or $3$.  By induction, $(w^2-1)$ divides $(q_{6k-1}-q_{6k+1})$.  Therefore, $(w+1)^3$ divides $q_{6k+3} \bmod 2$ if and only if $(w+1)^3$ divides $q_{6k-3} \bmod 2$.  Since $(w+1)^3$ does not divide $q_{-3} \bmod 2$, it is is not a factor for any $q_n \bmod 2$ where $3 | n$.
Lastly, by induction, $q_n(0)\equiv 1 \mod 2$ for all $n$.  This shows that $w$ is not a factor of $q_n \bmod 2$ which implies $\mbox{deg } g_i \geq 2$ for all $2\leq i\leq m$.

\end{proof}

Our proof will also require the following standard facts about the reduction of polynomials modulo primes and the factorization of ideals in number fields (for example, see \cite{Ko}, Sections 3.8 and 4.8). 

\begin{theorem} \label{num} Let $f(x) \in \Z[x]$ be an irreducible monic polynomial, $\alpha$ a root, and $k=\Q(\alpha)$ with ring of integers $\ring_k$.  Let $d_k$ denote the discriminant of $k$ and $\Delta(\alpha)$ the discriminant of $f$.  Let $p$ be a rational prime and $\bar{f}$ the reduction of $f$ modulo $p$.  
\begin{itemize}
\item[(i)] $\bar{f}$ decomposes into distinct irreducible factors if and only if $p$ does not divide $\Delta(\alpha)$.
\item[(ii)]  Suppose that $p$ does not divide $\Delta(\alpha) d_k^{-1}$ and $\bar{f}=\bar{f_1}^{e_1} \cdots \bar{f_m}^{e_m}$.  Then $p \ring_k = \pr_1^{e_1} \cdots \pr_m^{e_m}$.
\item[(iii)] Let $\pr_1, \cdots, \pr_m$ be the prime divisors of $p$ in $\ring_k$ with ramification indices $e_1,\cdots, e_m$, let $k_{\pr_i}$ be the completion of $k$ with respect to the valuation $v_i=e_i^{-1}v_{\pr_i}$ and let $\Q_p$ denote the completion of $\Q$ with respect to the valuation $v_p$.  Then the ramification index of $k_{\pr_i}$ over $\Q_p$ is equal to the ramification index of $\pr_i$ over $p$ in $k/\Q$. 
\end{itemize}
\end{theorem}

We also require the following two lemmas in our proof.

\begin{lemma} \label{lem1}Let $g(w) \in \Z[w]$ be an irreducible monic polynomial, $\alpha$ a root, and $k=\Q(\alpha)$ with ring of integers $\ring_k$.  Let $\bar{g}$ denote the reduction of $g$ modulo 2. Suppose further that
$$\bar{g} = \bar{g_1}^2 \prod_{i=2}^{m}\bar{g_i},$$
where $\bar{g_i}$ are relatively prime with $\mbox{{\rm deg} } \bar{g_1} =1$ and $\mbox{\rm deg } \bar{g_i} \geq 2$ for $2 \leq i \leq m$.  Then either
$$
2\ring_k = \pr_1 \cdots \pr_{m+1} \quad \mbox{ or }\quad  2\ring_k = \pr_1^2 \pr_2  \cdots \pr_{m}.
$$
\end{lemma}

\begin{proof} If the factorization of $g \bmod 2$ corresponds to the factorization of $2 \ring_k$, then we are done.  If not, then using Theorem \ref{num} (iii) and Hensel's lemma, we can determine the factorization of $2\ring_k$ using the $2$-adic factorization of $g$.  Since the residue class $\Z/2\Z$ is finite, the decomposition of $g$ into irreducible factors over $\Q_2$ can be accomplished in finitely many steps.  Consider the square-free part of $g \bmod 2$.  Then $g \bmod 2^l$ is square-free for all $l \geq 2$.   To see this, suppose that $g -u^2 q = 2^{l+1} h$ for some integer $l >1$ and polynomials $u,q, h \in \Z[w]$. Then $g-u^2q= 2(2^l h)$ would imply that $g \bmod 2$ is not square-free, which is a contradiction.  Also, by the same argument, the square-free part of $g \bmod 2^l$ will have no linear factors when $l>1$.  In fact, each factor of the square-free part of 
$g \bmod 2^l$ 
corresponds to exactly one factor of the square-free part of $g \bmod 2$. This shows that for each $i$, $2 \leq i \leq m$, there is a unique prime $\pr_i$ dividing 2 in $\ring_k$ corresponding to the factor $\bar{g_i}$.  Moreover, $\pr_i$ has ramification index $e_i=1$ for $2 \leq i \leq m$.  Now, there will be two primes (respectively one prime) dividing 2 corresponding to $\bar{g_1}$ in $\Q_2[w]$ if $\bar{g_1}^2 \mod 2^l$ factors (respectively does not factor) into distinct irreducible linear factors  for large enough $l$.  This finishes the proof.

\end{proof}

\begin{lemma} \label{notimag}  The polynomial $q_n$ has no quadratic factor that reduces to $(w+1)^2 \bmod 2$.
\end{lemma}
\begin{proof}  A quadratic monic polynomial $f$ such that $f \equiv (w+1)^2 \mod 2$ has the form $f(w)=w^2+2aw+(2b+1)$ for some integers $a, b$. 
(Since $q_n$ is monic, we can assume $f$ is as well.) The polynomial $f$ has discriminant $4(a^2-2b-1)$, so if $k_n=\Q(i)$ is defined by $f$, then $a^2 -2b-1=-d^2$, for some nonzero integer $d$.  One can prove by induction that $q_n(2)=1.$  This implies that $f(2)= \pm 1.$   If $f(2)= -1= 4+4a+2b+1$, then $b=-2a-3.$  But this implies that $-d^2 = a^2+4a+5= (a+2)^2+1 \geq 1$, which is a contradiction.  If $f(2)=1$, then $b=-2a-2$ and $-d^2=(a+3)(a+1)$.  This implies $a=-2$ and $b=2$, and so $f(x)=w^2-4w+5$.  But this gives a contradiction as 5 does not divide $q_n(0)$ which is either $-7$ or $-9.$
\end{proof}

\begin{proof}[Proof of Proposition \ref{noimag}]  Let $\rho$ denote the parabolic representation of $\pi_1(S^3 \setminus K_n)$ corresponding to the faithful discrete representation conjugated to be in the form as described in Equation~(\ref{eqfgh}) of the previous section and let $\Lambda_n(w)$ be the irreducible factor of $q_n(w)$ giving the representation corresponding to the complete structure. Denote the image group by $\Gamma_n$.  Then the trace field $k_n= k\Gamma_n=\Q(\beta_n)$ corresponds to some root $\beta_n$ of the polynomial $\Lambda_n(w)$.

If $3 \nmid n$, then by Prop.\ \ref{prop}, $\Lambda_n(w)$ has distinct factors modulo 2.  Therefore, by Theorem \ref{num}, 2 does not divide the discriminant $\Delta(\beta_n)$ of $\Lambda_n(w)$.  Since the discriminant $d_{k_n}$ of $\Q(\beta_n)$ divides $\Delta(\beta_n)$, it follows that 2 does not divide the discriminant of $k_n$.  Since the discriminant of $\Q(i)$ is -4, $\Q(i)$ cannot be a subfield of $k_n$.  This follows from standard facts about the behavior of the discriminant in extensions of number fields (see \cite{Ko}, Ch. 3, for example.) 

In the case $3 | n$, there are two situations by Lemma \ref{lem1}.  Let $\ring_{k_n}$ denote the ring of integers in  $k_n$.  If there is no ramified prime in $\ring_{k_n}$ dividing 2, then the argument follows as above.  If there is such a prime, then 
$$
2\ring_{k_n} = \pr_1^2 \pr_2  \cdots \pr_{m}.
$$
is the prime factorization of 2 in $\ring_{k_n}$.  We will suppose that $\Q(i) \subset k_n$ and derive a contradiction.  Now, the ring of integers of $\Q(i)$ is $\Z[i]$; moreover, the prime factorization of 2 in $\Z[i]$ is $\pri^2$, where $\pri=(1+i)\Z[i]$.  Since $\Q(i) \subset k_n$, it follows that $2$ divides the ramification index $e_j$ of each prime ideal $\pr_j$ dividing 2 in $k_n$.  If $k_n=\Q(i)$, then  $\Lambda_n$ is quadratic, but by Lemma \ref{notimag} $\Lambda_n \not\equiv (w+1)^2 \bmod 2$.  Therefore, $\Lambda_n \bmod 2$ has at least one factor corresponding to a prime $\pr$ dividing $2$ such that $\pr^2$ does not divide 2.  This gives the desired contradiction.
\end{proof}

\section{$\Q(\sqrt{-3})$ is not a subfield of $k_n$.\label{secQ3}}

In this section, we prove Proposition \ref{nosqrt3}.  Unless 
otherwise indicated, we will use ``$\equiv$'' to denote equivalence mod $3$ throughout this
section, although this reduction may occur in different rings.  

The argument that there is no $\Q(\sqrt{-3})$ subfield breaks into two cases as it is
convenient to use $q_n$ when $3 \mid n$ and $p_n$ otherwise.  

\subsection{Case 1}

Let $n$ be negative and odd with $3 \mid n$ and let $q_n \in \Z[w]$ be the polynomials defined in Section \ref{secprelim}.
\begin{proposition}\label{sqrt3}  Let $n \equiv 0 \bmod 3$.
If $n \equiv 3$ mod 4, then $w$ does not divide $q_{n}$ mod 3. 
If $n \equiv 1$ mod 4, then $w^2$ divides $q_{n}$ mod 3 but $w^3$ does not.
\end{proposition}

\begin{proof}
By induction, the constant term of $q_n$ is $-7$ if $n \equiv 3$ mod 4 and
$-9$ if $n \equiv 1$ mod 4. So, if $n \equiv 3$ mod 4, $w$ does not divide $q_n$ mod 3.

To see that $w^2$ divides $q_n$ for $n \equiv 1$ mod 4, note that, by induction, 
for such an $n$, $w^2$ divides
$$q_{n+2} - q_{n+6} = (w^2-1) q_{n+4} + q_{n+8} + w^2 q_{n+6}$$
since $n+4$ and $n+8$ are also $1$ mod $4$. It follows that $w^2$ divides
$$q_n = w^2 q_{n+2} - (w^2-1)q_{n+4} - (q_{n+2} - q_{n+6}).$$

Finally, we can argue that $w^3$ does not divide $q_n$ by noting that the
$w^2$ coefficient of $q_n$ is never $0$ mod $3$. Let $q_{n,k}$ denote the coefficient of $w^k$ in $q_n$. Then,
$$q_{n,2} = -q_{n+2,2} + q_{n+4,2} + q_{n+6,2} + q_{n+2,0} - q_{n+4,0}.
$$
We have already mentioned that the constant coefficients $q_{n+2,0}$ and
$q_{n+4,0}$ are either $-7$ or $-9$. So, we have a simple recursion for
the $w^2$ coefficients which shows  that they cycle through the values 
$2,2,2,1,1,1,2,2,2,1,1,1,\ldots$ modulo 3.
\end{proof}

Using the substitution $w = x+x^{-1}$, we can derive a closed form for a sequence of Laurent polynmomials related to $q_n$.  Letting $r_n (x)= q_n(x+x^{-1})$ and using the recursion relation for $q_n$, one can establish that $ r_{n}(x) =f_{(3-n)/2}(x)$ where 
$$f_k = x{[}x^{2k} + x^{-2k} + (x^3 + 4x^2 - 8 + 4x^{-2} + x^{-3}){]}/(x + 1)^2.$$
(We thank Frank Calegari, Ronald van Luijk, and Don Zagier for help in determining this closed form.)
In the ring $\Z[w][x]/(x^2-wx+1)\cong \Z[x,x^{-1}]$, we have $w = x+x^{-1}$.  (Note that $\mbox{deg }(x^{\mbox{\small deg } q_n} r_n(x)) = 2 \mbox{ deg } q_n$, i.e., the polynomial $x^{\mbox{deg } q_n} r_n \in \Z[x]$, or equivalently the numerator of $f_{(3-n)/2}$, indeed defines a quadratic extension of $k_n.$) However, it will be more convenient to work with the Laurent polynomials $f_k$ in the ring $\Z[x,x^{-1},(x+1)^{-1}] \supset \Z[w]$.  Since 0 and 1 are not roots of $q_n$, it suffices to look at the reduction of $q_n$ modulo 3 in this ring.  

\begin{lemma} \label{lem2}
If $3|n$, then 
$$q_{n+2} - q_{n+4} + (w^2-1)(q'_{n+2} - q'_{n+4}) \equiv 0.$$
\end{lemma}

\begin{proof}
Working modulo 3, we have that
$$
f_{k}  \equiv x(x+1)^{-2}[\left( x^{k}-x^{-k} \right)^2 +x^{3}+x^{2} + x +x^{-2} + x^{-3}]$$
and
$$
\frac{df_k}{dx}  \equiv (-x+1)(x+1)^{-3}\left( x^{k}-x^{-k} \right)^2-k(x+1)^2(x^{2k}-x^{-2k})
-x-1+x^{-2}+x^{-3}.$$
Moreover,
$$
 r'_{n} = \frac{dr_n}{dw}  = \frac{dr_n}{dx} \frac{ dx}{dw}= \frac{x^2}{x^2-1} \frac{df_{(n+3)/2}}{dx}.
$$
This gives
$$
(x^{2}+1+x^{-2}) r'_{n} \equiv  (x^2+1) \frac{df_{(n+3)/2}}{dx}.
$$

So, if $3|n$, after applying the substitution $w = x+1/x$ and using the above formulas, we get
\begin{eqnarray*}
&&q_{n+2} - q_{n+4} + (w^2-1)(q'_{n+2} - q'_{n+4})  \\
&=& r_{n+2}-r_{n+4} + (x^{2}+1+x^{-2})(r'_{n+2} - r'_{n+4}) \\
& \equiv & f_{(3-(n+2))/2} - f_{(3-(n+4))/2} \\
&& \mbox{ }
+ (x^2+1) {[} df_{(3-(n+2))/2}/dx - df_{(3-(n+4))/2}/dx {]} \\
& \equiv & (x+1)^{-2} {[}(x^{-(n+1)}+x^{n+1} -x^{1-n}-x^{n-1})(x^2+1) \\
&& \mbox{ } + (x^2-1)(x^{1-n}-x^{n-1}-x^{n+1}+x^{-(n+1)}){]} \equiv 0
\end{eqnarray*} 
as required.
\end{proof}

\begin{lemma} \label{lem3} When $3|n$, 
$$q_n - (1-w)q'_n \equiv -w$$
\end{lemma}

\begin{proof}
Since $q_n  =  (w^2-1)(q_{n+2} - q_{n+4}) + q_{n+6}$, 
then $$q'_n \equiv -w(q_{n+2}-q_{n+4}) + (w^2-1)(q'_{n+2} - q'_{n+4}) + q'_{n+6}.$$
Using the previous lemma we have $q_{n} - (1-w)q'_n  \equiv q_{n+6} - (1-w)q'_{n+6}$ and the proof follows by induction. 
\end{proof}

\begin{lemma}\label{lem4}
Let $3|n$ and $n \equiv 1$ mod 4. There is no quadratic factor of $q_n$ that reduces to $w^2$ mod 3.
\end{lemma}

\begin{proof}
We've seen that the constant term of $q_n$ is $-9$. So the constant term of
such a quadratic factor is $\pm 3$ or $\pm 9$. Since $q_n$ is monic, we can assume that such a factor is as well. So, if such a factor exists, it's of the form $w^2 + 3aw + b$ where $b \in \{ \pm3 , \pm 9\}$.

Now, by induction, $q_n(2) = 1$, for all $n$. So, the quadratic factor
must evaluate to $\pm1$ when $w = 2$. This shows that the factor is one of the following: $w^2-3w+3$, $w^2-3$, $w^2 -6w+9$, or $w^2 + 3w -9$.

We can also argue, by induction, that
$q_n(1) = -4$ when $3|n$ and $n \equiv 1$ mod 4. So, the quadratic factor must 
divide $4$ when $w = 1$ is substituted. This eliminates $w^2 + 3w-9$ as a candidate.

Similarly, the requirement that $q_n(-1) = -8$ leaves only $w^2-3$ as a candidate. However, an induction argument shows that $q_n(\sqrt{3}) = -12 + 6 \sqrt{3}$ when $3|n$ and $n \equiv 1$ mod 4. So, $w^2-3$ is also not a quadratic factor.  Thus, as required, $q_n$ has no quadratic factor that reduces to $w^2$ mod 3. 
\end{proof}
We now have the ingredients to prove the following:

\begin{proposition}\label{3divn} Let $K_n$ denote the $(-2,3,n)$ pretzel knot with trace field $k_n$.  Suppose further that $3 | n$. Then $\Q(\sqrt{-3})$ is not a subfield of $k_n$.
\end{proposition}

\begin{proof}
As in the proof of Proposition~\ref{noimag}, let $\rho_0$ denote the parabolic representation of $\pi_1(S^3 \setminus K_n)$ corresponding to the discrete faithful representation, $\Lambda_n(w)$ the irreducible factor of $q_n(w)$ corresponding to this representation, and $\Gamma_n$ the image group.  Then $k_n= k\Gamma_n=\Q(\beta_n)$.  By Lemma~\ref{lem3}, the gcd of $q_n$ and $q'_n$ modulo $3$ is either $1$ or $w$ .

Since $w$ is not a factor of $q_n$ when $n \equiv 3$ mod 4, it follows that
$q_n$ and $q'_n$ have no common factors modulo $3$ in case both
$n \equiv 3$ mod 4 and $3|n$.  Therefore, $\Lambda_n$ has distinct irreducible factors mod $3$ and, by Theorem~\ref{num}, we conclude that $3$ doesn't divide the discriminant of $k_n$ so that  $\Q(\sqrt{-3})$ cannot be a subfield of $k_n$.

On the other hand, if $3|n$ and $n \equiv 1$ mod 4, then by Proposition~\ref{sqrt3} and Lemma \ref{lem3}, we deduce that the gcd of $q_n$ and $q_n'$ is $w$ and moreover that 
$$
q_n \equiv  w^2 \prod_{i=2}^m g_i(w),
$$
where $g_i$ are relatively prime and irreducible.  Since the behavior of the prime ideal 3 in the ring of integers of $\Q(\sqrt{-3})$ is identical to that of the ideal 2 in the ring of integers of $\Q(i)$ and since these fields are both quadratic imaginary, we can apply the same argument used for the case $3 \mid n$ in the proof of Proposition \ref{noimag} replacing Lemma \ref{notimag} with Lemma \ref{lem4}.
\end{proof}

\subsection{Case 2}

Let $n$ be negative and odd with $3 \nmid n$ and 
let $p_n \in \Z[v]$ be the polynomials defined in Section~\ref{secprelim}. We will
argue that $p_n$ has no repeated roots modulo $3$. It will then follow
from Theorem~\ref{num} that $3$ does not divide the discriminant of
the trace field $k_n$ so that $\Q(\sqrt{-3})$ cannot be a subfield.

A straightforward induction shows that the following is a closed form
for $p_n$ modulo $3$:
\begin{equation} \label{eqab}
p_n \equiv {[} (a+b)^k-(a-b)^k-(a+b)^{k+2}+(a-b)^{k+2} {]}/(vb)
\end{equation} 
where $a= (v^2+v-1)$, $b^2=(v^2-1)(v^2-v-1)$, and $k = (1-n)/2$.
This formula requires a little interpretation. First,
note that it can be rearranged as
\begin{equation}\label{eqabsum}%
p_n \equiv \frac{-1}{v} \left( \sum_{ \begin{smallmatrix} 1 \leq i \leq k \\ i \mbox{ {\tiny odd}}\end{smallmatrix} } \binom{k}{i} a^{k-i}b^{i-1} - 
\sum_{ \begin{smallmatrix} 1 \leq i \leq k +2 \\ i \mbox{ {\tiny odd}}\end{smallmatrix} } \binom{k+2}{i} a^{k+2-i}b^{i-1} \right).
\end{equation}
This shows that $p_n =v^{-1} g_n$, where $g_n \in Z[v]$.  Furthermore, the constant term of $p_n$ is $2^{-(n+1)/2}$, so that $v$ is not a factor of $p_n$ modulo $3$.  Therefore, $g_n \equiv -v p_n$ where 
\begin{equation} \label{eqgab}
g_n = {[} (a+b)^k-(a-b)^k-(a+b)^{k+2}+(a-b)^{k+2} {]}/b.
\end{equation}

Thus, our goal is to argue that $g_n$ has no repeated
factors modulo $3$. Let $\F_3 \cong \Z/3\Z$ be the field of 
three elements and fix an algebraic closure $\Fb_3$. We will
take advantage of the fact that $f, g \in \F_3[v]$ have a common
factor if and only if $f$ and $g$ have a common
root in $\Fb_3$. For the sake of convenience, we will often use 
the same symbol, $g_n$, $a$, etc.\ to represent both the polynomial
in $\Z[v]$ and its reduction mod $3$ in $\F_3[v]$.

We first examine when $a$ or $b^2$ can have common factors with $g_n$.

\begin{lemma}
The polynomials $b^2$ and $g_n$ in $\F_3[v]$ have no common factor. 
\end{lemma}

\begin{proof}
By induction (using the recurrence given in Section~\ref{secprelim}), 
$p_n(1) \equiv -1$ and $p_n(-1) \equiv (-1)^{-(n+1)/2}$ for all odd and
negative $n$. So neither $v = 1$ nor $v=-1$ is a root of $p_n$ and, hence, neither $(v-1)$ nor $(v+1)$ is a factor $g_n$ in $\F_3[v]$. 

Using the form of $p_n$ given by Equation~(\ref{eqabsum}) and
evaluating at a root $v_0$ of $(v^2-v-1)$ (i.e., working in $\Fb_3$), 
the powers of $b^2$ become zero and we're left with
$g_n(v_0) \equiv a^{k-1} (a^2 - 1) \equiv a^{k-1} v_0(v_0+1)(v_0-1)$. But, at a root $v_0$ of $(v^2-v-1)$, $a$ becomes $-v_0$. Since neither $0$, nor $\pm 1$ is a root of $v^2-v-1$, $g_n(v_0) \not\equiv 0$.
Thus, $(v^2-v-1)$ also has no common factor with $g_n$ in $\F_3[v]$.
\end{proof}

\begin{lemma}
The irreducible polynomial $a = v^2+v-1$ is a factor of $g_n$ mod 3 if and only if $n \equiv 1$ mod 4. However, it is never a repeated factor.
\end{lemma}

\begin{proof}
That $a$ is a factor of $p_n$ (hence of $g_n$) if and only if $n \equiv 1$ mod 4 is easily verified by induction. (Note that $a \equiv v^2+v+2$ appears as part of the recursion equation). 

Suppose $n \equiv 1$ mod 4 (so that $k$ is even) and write $g_n$ as a sum:
\begin{eqnarray*}
g_n & = & \left( \sum_{ \begin{smallmatrix} 1 \leq i \leq k \\ i \mbox{ {\tiny odd}}\end{smallmatrix} } \binom{k}{i} b^{k-1-i}a^{i} - 
\sum_{ \begin{smallmatrix} 1 \leq i \leq k +2 \\ i \mbox{ {\tiny odd}}\end{smallmatrix} } \binom{k+2}{i}
 b^{k+1-i}a^{i} \right) \\
& = & -a^3
\sum_{ i \geq 3, \mbox{ {\tiny odd}}} \left[ \binom{k}{i} b^{k-1-i}a^{i-3} -
\binom{k+2}{i} b^{k+1-i}a^{i-3} \right] - a b^{k-2}(1-b^2).
\end{eqnarray*}
Thus, $a^2$ and $g_n$ share a factor in $\F_3[v]$ only if $a$ and
$b^{k-2}(1-b^2)$ do. 

If $a$ and $b^{k-2}(1-b^2)$ have a common factor, then $a$ shares
a root in $\Fb_3$ with $b^{k-2}$ or
$(1-b^2) \equiv -v(v^3 -v^2 + v+1)$.
However, if $v_0$ is a root of $a$, then $v_0 \neq \pm 1$ because $a(1) = 1^2 +1 -1 = 1 \neq 0$ and $a(-1) = -1 \neq 0$.
Also, at a root $v_0$ of $a$, $v_0^2-v_0-1$ becomes $v_0$ which is not zero
since $a(0) = -1 \neq 0$. So, at a root of $a$, the factor $b^{k-2}$ is
not zero.

As for $(1-b^2)$, evaluated at a root $v_0$ of $a$, $v_0^3 - v_0^2 + v_0 + 1 \equiv v_0-1$. Thus neither this factor of $(1-b^2)$ nor the other factor, $v$, is zero at $v_0$, since, again, $v_0 \neq 0,1$.
Thus, $(1-b^2)$ and $a$ also share no root. It follows that $a^2$ does not divide $g_n$ modulo 3. 
\end{proof}

\begin{proposition} \label{propQ3case2}
Let $n$ be odd and negative with $3 \nmid n$. Then $g_n$ and $g'_n$ have no common factor in $\F_3[v]$.
\end{proposition}

\begin{proof}
Suppose, for a contradiction, that $g_n$ and $g'_n$ have a common factor 
in $\F_3[v]$.  Then they will have a common root $v_0 \in \Fb_3$.
As we have noted, $v$ is not a factor of $p_n$, so, it is not a
common factor of $g_n$ and $g'_n$. Thus,
$v_0 \neq 0$, and the lemmas show that $v_0$ is not a root of
$a$ or any factor of $b^2$. In particular, $v_0 \neq \pm 1$. 

Most of our calculations in this proof will take place in $\Fb_3$
and we will frequently evaluate polynomials at $v_0$ to get a
value in $\Fb_3$. To facilitate our calculations, we fix a square root
of $b^2(v_0)$ and call it $b$. Since $v_0$ is not a root of $b^2$, 
$b$ is not zero.

Note that $(a+b)^k-(a-b)^k$ is not zero at $v = v_0$. For otherwise, evaluated at $v_0$, we would have $(a+b)^k = (a-b)^k$. On the other hand, since $v_0$ is a zero of $g_n$, we
have also that $(a+b)^{k+2}-(a-b)^{k+2} = 0$, or, equivalently, $(a+b)^{k
+2} = (a-b)^{k+2}$ when $v = v_0$. It follows that, either 
both $(a+b)^k$ and $(a-b)^k$ are zero at $v_0$, or else, 
$(a+b)^2 = (a-b)^2$ when evaluated at $v_0$.  
Now, if $(a+b)^2 = (a-b)^2$, we deduce that $ab$ is zero at $v_0$, a contradiction. On the other hand, if both
$(a+b)^k$ and $(a-b)^k$ are zero, then $a+b$ and $a-b$ are too, which 
again implies $v_0$ is a root of $a$, a contradiction.

Now, at $v = v_0$, we can write
\begin{eqnarray*}
g_n & = &
{[} (a+b)^k-(a-b)^k-(a+b)^{k}(a+b)^2+(a-b)^{k}(a-b)^2 {]}/b \\
& \equiv &
{[} (a+b)^k(1-a^2-b^2+ab)-(a-b)^k(1-a^2-b^2-ab){]}/b \\
& \equiv &
(v_0+1)^3(v_0-1){[} (a+b)^k-(a-b)^k{]}/b +a{[}(a+b)^k+(a-b)^k{]}
\end{eqnarray*}
Thus, evaluating at $v_0$, we will have
$$-b \frac{(a+b)^k+(a-b)^k}{ (a+b)^k-(a-b)^k} 
= (v_0+1)^3(v_0-1)/a.$$

Since $k = (1-n)/2$ in Equation~(\ref{eqab}),
we can assume that $k \equiv 0$ or $1$.
Our goal is to derive a contradiction in both cases.

Suppose first that $k \equiv 0$. Then the derivative is
\begin{eqnarray*}
g'_n & \equiv & - g_n(b'/b) 
+
{[} k(a+b)^{k-1}(a'+b')-k(a-b)^{k-1}(a'-b') \\
&& \mbox{ } - (k+2)(a+b)^{k+1}(a'+b')+(k+2)(a-b)^{k+1}(a'-b') {]}/b \\
& \equiv & \left( g_n(v^3-v+1) 
+
(v+1)^3 b {[}(a+b)^{k}-(a-b)^{k}{]} \right. \\
&& \mbox{ } \left.  + (v^5+v^4+v-1){[}(a+b)^{k}+(a-b)^{k} {]} \right)/b^2
\end{eqnarray*}
where the first line suggests an algebraic means of deriving the formula
given in the second line. Again, the $b^2$ in the denominator of the 
second line is only there for the sake of presenting a simple formula; 
it cancels to leave a polynomial $g'_n \in \F_3[v]$.

As above, we may assume that we are evaluating these expressions at a
common zero $v_0$ of $g_n$ and $g'_n$, which is not a zero of $v+1$ nor of
$(a+b)^k-(a-b)^k$.
It follows that the factor $v^5+v^4+v-1$ is also not zero at $v_0$. 
So, at $v_0$ we have
$$-b \frac{(a+b)^k+(a-b)^k}{ (a+b)^k-(a-b)^k} 
= \frac{b^2(v_0+1)^3}{v_0^5+v_0^4+v_0-1}.$$

Comparing our two expressions for $-b \frac{(a+b)^k+(a-b)^k}{ (a+b)^k-(a-b
)^k} $ we see that
$$ (v_0-1) (v_0^5+v_0^4+v_0-1) = ab^2  \Rightarrow  v_0 + 2 = 0.$$

So, the only possibility for a common zero is $v_0 = -2 \equiv 1$. 
However, we have already noted that $v_0 \neq 1$. The contradiction
completes the argument in the case $k \equiv 0$.

The argument for the $k \equiv 1$ case is similar and based on using $g_n$ 
and $g'_n$ to derive two different expressions for
$$-b \frac{(a+b)^{k-1}+(a-b)^{k-1}}{ (a+b)^{k-1}-(a-b)^{k-1}}.$$
\end{proof}\begin{proposition}\label{3ndivn} Let $K_n$ denote the $(-2,3,n)$ pretzel knot with trace field $k_n$.  Suppose further that $3 \nmid n$. Then $\Q(\sqrt{-3})$ is not a subfield of $k_n$.
\end{proposition}

\begin{proof}
As in the proof of Proposition~\ref{3divn}, let $\rho_0$ denote
the discrete faithful representation of $\pi_1(S^3 \setminus K_n)$,
let $\Lambda_n(v)$ be
the irreducible factor of $p_n$ giving the representation
corresponding to the complete structure, and $\Gamma_n$ the 
image group. Then $k_n = k \Gamma_n = \Q(\alpha_n)$ for
some root $\alpha_n$ of $\Lambda_n$.

Since, by Proposition~\ref{propQ3case2}, $g_n$ and $g'_n$ have no common factors in $\F_3[v]$,
$g_n$ has distinct irreducible factors modulo $3$. Since
$v$ is not a factor of $p_n$ and $g_n \equiv v p_n$, it follows
that $p_n$ and, therefore, $\Lambda_n$ also have distinct irreducible
factors modulo $3$. By Theorem~\ref{num}, $3$ does not divide
the discriminant of $k_n$ so that $\Q(\sqrt{-3})$, having discriminant
$-3$, cannot be a subfield.
\end{proof}

\begin{proof}[Proof of Proposition \ref{nosqrt3}] This follows immediately from Propositions \ref{3divn} and \ref{3ndivn}.
\end{proof}

\section{Commensurability classes of Montesinos knots \label{secMont}}

Let $K$ be a hyperbolic Montesinos knot and $M = S^3 \setminus K$ its
complement.
According to Theorem~\ref{thmRW}, we can ensure that $M$ is the only knot complement in its commensurability class
by showing that $K$ enjoys the following three properties.
\begin{enumerate}
\item $K$ has no lens space surgeries.
\item Either $K$ has no symmetries, or it has only a strong inversion
and no other symmetries.
\item $K$ admits no hidden symmetries.
\end{enumerate}

The first two properties are well understood. According to \cite{M}, $K$ has no 
non-trivial cyclic, and hence no lens space, surgeries unless $K$ is the 
$(-2,3,7)$ pretzel knot or $K$ is of the form $M(x, 1/p,1/q)$ with 
$x \in \{ -1 \pm 1/2n, -2 + 1/2n \}$ and $n$, $p$, and $q$ positive integers. 
(No examples of a $M(x, 1/p, 1/q)$ knot with a lens space surgery
are known, but it is remains an open problem to show that there are none.)
As for the second property, the symmetries of Montesinos knots are
classified in \cite{BZ,S}. Thus, for a broad class of Montesinos knots,
understanding the commensurability class comes down to understanding 
hidden symmetries.

For example, if we restrict to the class of three tangle pretzel knots,
we have the following:

\begin{theorem} \label{thmpqr}
Let $K$ be a $(p, q, r )$ pretzel knot 
with $|p|, |q|, |r| > 1$,
$\{p,q,r \} \not\in \{ \{-2,3,5\}, \{-2,3,7 \} \}$,
and exactly two of $p,q,r$ odd with those two unequal.
If $K$ has no hidden symmetries, then $S^3 \setminus K$ is 
the only knot complement in its commensurability class.
\end{theorem}
\begin{proof}
The conditions on $p,q,r$ ensure that $K$ is a hyperbolic knot \cite{K}
with a strong inversion. By \cite[Theorem~1.3]{BZ} and \cite[Theorem~6.2]{S}, $K$ has no other symmetries. By \cite[Theorem~1.1]{M}, $K$ has no lens space surgery. So, if in addition $K$ has no
hidden symmetries, then by Theorem~\ref{thmRW}, $S^3 \setminus K$ is the unique knot complement in its class.
\end{proof}

For pretzel knots up to ten crossings, we can show the following:
\begin{theorem}
Let $K$ be a $(p,q,r)$ pretzel knot with $p,q,r$ as in Theorem~\ref{thmpqr} and with at most ten crossings.
Then $S^3 \setminus K$ is the only knot complement in its commensurability class.
\end{theorem}

\begin{remark}Using the computer software {\em Snap} we can extend this to twelve crossings; according to~\cite{GHH}, none of the pretzel knots of the type described in Theorem~\ref{thmpqr} with twelve or fewer crossings has a hidden symmetry. 
\end{remark}

\begin{proof}
By Theorem~\ref{thmgrand}, the theorem holds if
$K$ is a $(-2,3,n)$ pretzel knot. The only other candidates of ten or fewer crossings are
$(2,3,5)$ ($10_{46}$ in the tables), $(2,3,-5)$ ($10_{126}$), and 
$(-3,3,4)$ ($10_{140}$). We will show that each of these three has no hidden
symmetries by demonstrating that the trace field has no $\Q(i)$ nor $\Q(\sqrt{-3})$ subfield.
Indeed, in each case we will show that the trace field is an odd degree extension of $\Q$ and,
therefore, admits no quadratic subfield.

The $(2,3,5)$ pretzel knot has fundamental group 
$$ \Gamma_{2,3,5} \cong \langle f,g,h \mid hfhfg^{-1}=fhfg^{-1}f, gf^{-1}ghghg=f^{-1}ghghgh \rangle $$
and we can use the same parametrisation of the parabolic 
$\SLC$-representations 
as in Equation~(\ref{eqfgh}). Then,
the lower right entry of $\rho(hfhfg^{-1}) - \rho(fhfg^{-1}f)$ is $v(u(u^2-1) + wv[u^2-1 + uv(u^2-2)])$.
Since $v \neq 0$ ($v=0$ would imply $\rho$ is not faithful), we must have $w = u(1-u^2)/(v[u^2-1 + uv(u^2-2)])$.  On making this substitution,
we see that $\rho$ will satisfy the first relation if $u-1+v(u^2-u-1)=0$ or, equivalently,
$v = (u-1)/( 1+u-u^2)$.  The second relation will then be satisfied if $u$ is a root of the irreducible polynomial
\begin{eqnarray*}
p_{2,3,5} & = & u^{17}-3u^{16}-5u^{15}+18u^{14} + 14u^{13}-41u^{12}-46u^{11} +47u^{10}\\
&& \mbox{ }+104u^9-17u^8-114u^7-40u^6+56u^5+50u^4-8u^3-11u^2-2u+1.
\end{eqnarray*}
Note that for any root of $u$ of $p_{2,3,5}$, $u^2-1 + uv(u^2-2) \neq 0$ where $v =(u-1)/(1+u-u^2)$. Therefore the substitution $w = u(1-u^2)/(v[u^2-1 + uv(u^2-2)])$ is always defined and $p_{2,3,5}$ is indeed the Riley polynomial for the $(2,3,5)$ pretzel knot.   Thus, the discrete faithful representation $\rho_0$ corresponds to a root of $p_{2,3,5}$ and the trace field is a degree 17 extension of $\Q$.

The fundamental group of the $(2,3,-5)$ pretzel knot is 
$$ \Gamma_{2,3,-5} \cong \langle f,g,h \mid hfhfg^{-1}=fhfg^{-1}f, hghghf=ghghfg \rangle $$
so that we can satisfy the first relation using the same substitutions as for the $(2,3,5)$ pretzel knot.
The second relation will also be satisfied provided $u$ is a root of the irreducible polynomial
$$p_{2,3,-5} = u^{11}-3u^{10}-3u^9+12u^8+7u^7-18u^6-19u^5+13u^4+21u^3-u^2-7u+1.$$
So, the degree of the trace field is 11.

For the $(-3,3,4)$ knot we have that
$$ \Gamma_{-3,3,4} \cong \langle f,g,h \mid g^{-1}f^{-1}gfg=h^{-1}f^{-1}hfh, h^{-1}fhfhg^{-1}h=fhg^{-1}hg^{-1}hg \rangle. $$
In this case it's convenient to alter the parametrisation slightly:
$$
\rho(f) = \left(
\begin{array}{cc} 
1-u & -u/v \\
uv & 1+u 
\end{array} \right) , 
\rho(g) = \left(
\begin{array}{cc} 
1   & 0 \\
w^2 & 1 
\end{array} \right) ,
\mbox{ and } \rho(h) = \left(
\begin{array}{cc} 
1 & -1 \\
0 & 1 
\end{array} \right).
$$
The upper left entry of $\rho(g^{-1}f^{-1}gfg)-\rho(h^{-1}f^{-1}hfh)$ is
$$\frac{-u}{v^2} [u(v^4-v^3+vw^2+w^4) - v(v^2+w^2)]$$ which suggests
setting $u = v(v^2+w^2)/(w^4+vw^2-v^3+v^4)$. On making this
substitution, we see that the first relation will be satisfied provided
$v = w(w+1)/(w-1)$. Then, the second relation depends on $w$ satisfying
the irreducible polynomial $$p_{-3,3,4} = w^7-w^6 + 7w^5 -3 w^4 +12 w^3 +2 w^2 +4 w +2$$
so that the trace field is of degree 7 over $\Q$.
\end{proof}
\bibliographystyle{plain} 
\bibliography{MMbib}
		
\end{document}